\documentclass[a4paper,11pt,notitlepage]{article}
\usepackage{epsfig}
\usepackage{latexsym}
\usepackage{amssymb}
\usepackage[T1]{fontenc}
\usepackage{latexsym}		
\usepackage{amsmath}
\usepackage{amsthm}
\usepackage{graphics}
\usepackage[bolddef,boldvec]{niklas}
\theoremstyle{remark}

\theoremstyle{plain}
\newtheorem{theorem}{Theorem}[section]
\newtheorem{lemma}[theorem]{Lemma}
\newtheorem{proposition}[theorem]{Proposition}
\newtheorem{definition}{Definition}[section]
\newtheorem{example}{Example}[section]
\newtheorem{corollary}[theorem]{Corollary}
\newtheorem{problem}[theorem]{Problem}

\newcommand{\col}[1]{{\bf #1}}

\newcommand{\fix}   {\ensuremath{\mathrm{fix}}}
\newcommand{\expn}{\mathrm{exp}_{n}}

\begin{document}
\title{Enumeration of derangements with descents in prescribed positions}
\author{Niklas Eriksen, Ragnar Freij, Johan Wästlund\\ 
\small Department of Mathematical Sciences,\\[-0.8ex]
\small Chalmers University of Technology and University of Gothenburg,\\[-0.8ex] \small S-412 96 Göteborg, Sweden\\
\small \texttt{ner@chalmers.se  ragnar.freij@chalmers.se  wastlund@chalmers.se}}
\date{\small \today}

\maketitle

\begin{abstract}
We enumerate derangements with descents in prescribed positions. A generating function was given by Guo-Niu Han and Guoce Xin in 2007. We give a combinatorial proof of this result, and derive several explicit formulas. To this end, we
consider fixed point $\lambda$-coloured permutations, which are easily
enumerated. Several formulae regarding these numbers are given, as
well as a generalisation of Euler's difference tables. We also prove that except in a trivial special case, if a permutation $\pi$ is chosen uniformly among all permutations on $n$ elements, the events that $\pi$ has descents in a set $S$ of positions, and that $\pi$ is a derangement, are positively correlated.
\end{abstract}

In a permutation $\pi \in \Sn$, a \emph{descent} is a position $i$
such that $\pi_i > \pi_{i+1}$, and an \emph{ascent} is a position
where $\pi_i < \pi_{i+1}$. A \emph{fixed point} is a position $i$
where $\pi_i = i$. If $\pi_i > i$, then $i$ is called an
\emph{excedance}, while if $\pi_i < i$, $i$ is a
\emph{deficiency}. Richard Stanley \cite{Sta2006} conjectured that
permutations in $\SymGroup{2n}$ with descents at and only at odd
positions (commonly known as \emph{alternating permutations}) and $n$
fixed points are equinumerous with permutations in $\Sn$ without fixed
points, commonly known as \emph{derangements}.   

The conjecture was given a bijective proof by Chapman and Williams in
2007 \cite{ChaWil2007}. The solution is quite 
straightforward: If $\pi \in \SymGroup{2n}$ and $F \subseteq [2n]$ is
the set of fixed points, then removing the fixed points gives a
permutation $\tau$ in $\SymGroup{[2n]\setminus F}$ without fixed
points, and $\pi$ can be easily reconstructed from $\tau$.   

For instance, removing the fixed points in $\pi = 326451$ gives $\tau
= 361$ or $\tau = 231$ if we reduce it to $\SymGroup{3}$. To recover
$\pi$, we note that the fixed points in the first two descents must be
at the respective second positions, $2$ and $4$, since both $\tau_1$
and $\tau_2$ are excedances, that is above the fixed point diagonal
$\tau_i = i$. On the other hand, since $\tau_3 < 3$, the fixed point
in the third descent comes in its first position, $5$. With this
information, we immediately recover $\pi$. 

Alternating permutations are permutations which fall in and only in blocks of length two. A natural generalisation comes by considering permutations which fall in blocks of lengths $a_1, a_2, \ldots, a_k$ and have $k$ fixed points (this is obviously the maximum number of fixed points, since each descending block can have at most one). These permutations are in bijection with derangements which descend in blocks of length $a_1-1, a_2-1, \ldots, a_k-1$, and possibly also between them, a fact which was proved by Guo-Niu Han and Guoce Xin \cite{HanXinUnpub}.

In this article we compute the number of derangements which have descents in prescribed blocks and possibly also between them. A generating function was given by Han and Xin using a representation theory argument. We start by computing the generating function using simple combinatorial arguments (Section~\ref{sc:genfkn}), and then proceed to extract a closed formula in Section~\ref{sc:clofor}. 

Interestingly, this formula, which is a combination of factorials, can also be written as the same combination of an infinite family of other numbers, including the derangement numbers. We give a combinatorial interpretation of these families as the number of \emph{fixed point $\lambda$-coloured} permutations. 

For a uniformly chosen permutation, the events that it is a derangement and that its descent set is included in a given set are not independent. We prove that except for the permutations of odd length with no ascents, these events are positively
correlated. In fact, we prove that the number of permutations which are fixed point free when sorted decreasingly in each block is larger when there are few and large blocks, compared to many small blocks. The precise statement is found in Section \ref{sec:Correlation}. 

Finally, in Section \ref{sec: Euler}, we generalise some results concerning Euler's difference triangles from \cite{Rak2007} to fixed point $\lambda$-coloured permutations, using a new combinatorial interpretation. This interpretation is in line with the rest of this article, counting permutations having an initial descending segment and $\lambda$-coloured fixed points to the right of the initial segment. In addition, we also derive a relation between difference triangles with different values of $\lambda$. 

There are many papers devoted to counting permutations with prescribed descent sets and fixed points, see for instance
\cite{DesWac1993,GesReu1993} and references therein. More recent related papers include \cite{CorGesSavWil2007}, where Corteel et al.~considered the distribution of descents and major index over permutations without descents on the last $i$ positions, and \cite{Cho2008}, where Chow considers the problem of enumerating the involutions with prescribed descent set.

\section{Definitions and examples}
Let $[i, j] = \{i, i+1, \ldots, j\}$ and $[n] = [1, n]$. We think of $[n]$ as being decomposed into blocks of lengths
$a_1,\dots, a_k$, and we will consider permutations that decrease within these blocks. The permutations are allowed to decrease or increase in the breaks between the blocks. 

Consider a sequence $\vc{a}=(a_1, a_2, \dots, a_k)$ of nonnegative integers, with $\sum_i a_i=n$, and let $c_j=\sum_{i=1}^j
a_i$. We denote by $A_j$ the $j$:th block of $\vc a$, that is the set $A_j = [c_{j-1}+1, c_j] \subseteq [n]$.

Throughout the paper, $k$ will denote the number of blocks in a given composition. We let $\SymGroup{\vc{a}}\subseteq \Sn$ be the set of permutations that have descents at every place within the blocks, and may or may not have descents in the breaks between the blocks. In particular $\Sn = \SymGroup{(1,1,\dots,1)}$.

\begin{example} If $n=6$ and $\vc{a}=(4,2)$, then we consider permutations that are decreasing in positions 1--4 and in positions 5--6. Such a permutation is uniquely determined by the partition of the numbers 1--6 into these blocks, so the total number of such permutations is $$\binom{6}{4, 2} = 15.$$ Of these 15 permutations, those that are derangements are \begin{eqnarray} \notag 6543|21\\  \notag 6542|31\\ \notag 6541|32\\ \notag 6521|43\\ \notag 5421|63\\ \notag 5321|64\\ \notag 4321|65 \end{eqnarray}     
\end{example}

We define $D(\vc{a})$ to be the subset of $\SymGroup{\vc{a}}$ consisting of derangements, and our objective is to
enumerate this set. For simplicity, we also define $D_n=D(1,\dots,1)$. 

For every composition $\vc{a}$ of $n$, there is a natural map $\Phi_{\vc{a}}:\Sn\rightarrow S_{\vc{a}}$, given by simply sorting the entries in each block in decreasing order. For example, if $\sigma=25134$, we have $\Phi_{(3,2)}(\sigma)=52143$. Clearly each fiber of this map has $a_1!\ldots a_k!$ elements. 

The following maps on permutations will be used frequently in the paper.

\begin{definition}
For $\sigma \in \Sn$, let $\phi_{j,k}(\sigma) = \tau_1 \ldots 
\tau_{j-1} k \tau_j \ldots \tau_n$, where 
\[
\tau_i=\left\{ \begin{array}{ll}
\sigma_i & \textrm{if $\sigma_i < k$}\\
\sigma_{i}+1 & \textrm{if $\sigma_i \geq k$}
\end{array}\right.
\]
Similarly, let $\psi_j(\sigma)=\tau_1\ldots \tau_{j-1}  \tau_{j+1} \ldots \tau_n$ where
\[
\tau_i=\left\{ \begin{array}{ll}
\sigma_i & \textrm{if $\sigma_i < \sigma_j$;}\\
\sigma_{i}-1 & \textrm{if $\sigma_i > \sigma_j$.}
\end{array}\right. 
\]
\end{definition}
Thus, $\phi_{j, k}$ inserts the element $k$ at position $j$, increasing elements larger than $k$ by one and shifting elements to the right of position $j$ one step further to the right. The map $\psi_j$ removes the element at position $j$, decreasing larger elements by one and shifting those to its right one step left. 

We will often use the map $\phi_j = \phi_{j, j}$ which inserts a fixed point at position $j$. The generalisations to a set $F$ of fixed points to be inserted or removed are denoted $\phi_F(\sigma)$ and $\psi_F(\sigma)$, inserting elements in increasing order and removing them in decreasing order.

The maps $\phi$ and $\psi$ are perhaps most obvious in terms of permutation matrices. For a permutation $\sigma\in
\Sn$, we get $\phi_{j,k}(\sigma)$ by adding a new row below the $k$:th one, a new column before the $j$:th one, and an entry 
at their intersection. Similarly, $\psi_j(\sigma)$ is obtained by deleting the $j$:th column and the $\sigma_j$:th row. 

\begin{example} 

We illustrate by showing some permutation matrices. For $\pi=21$ and
$F=\{1,3\}$, we get

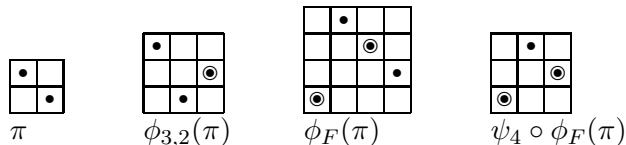
\begin{picture}(250,60)(-70, 0)
\multiput(0,10)(10,0){3}{\line(0,1){20}}
\multiput(0,10)(0,10){3}{\line(1,0){20}}
\put(5,25){\circle*{3}}
\put(15,15){\circle*{3}}
\put(0,0){$\pi$}

\multiput(50,10)(10,0){4}{\line(0,1){30}}
\multiput(50,10)(0,10){4}{\line(1,0){30}}
\put(55,35){\circle*{3}}
\put(65,15){\circle*{3}}
\put(75,25){\circle*{3}}
\put(75,25){\circle{5}}
\put(50,0){$\phi_{3,2}(\pi)$}

\multiput(110,10)(10,0){5}{\line(0,1){40}}
\multiput(110,10)(0,10){5}{\line(1,0){40}}
\put(115,15){\circle*{3}}
\put(115,15){\circle{5}}
\put(125,45){\circle*{3}}
\put(135,35){\circle*{3}}
\put(135,35){\circle{5}}
\put(145,25){\circle*{3}}
\put(110,0){$\phi_F(\pi)$}

\multiput(180,10)(10,0){4}{\line(0,1){30}}
\multiput(180,10)(0,10){4}{\line(1,0){30}}
\put(185,15){\circle*{3}}
\put(185,15){\circle{5}}
\put(195,35){\circle*{3}}
\put(205,25){\circle*{3}}
\put(205,25){\circle{5}}
\put(180,0){$\psi_4\circ\phi_F(\pi)$}
\end{picture}

where inserted points are labeled with an extra circle.
\end{example}

\section{A generating function} \label{sc:genfkn}

Guo-Niu Han and Guoce Xin gave a generating function for $D( a)$
(\cite{HanXinUnpub}, Theorem~9). In fact they proved this generating
function for another set of permutations, equinumerous to $D(\vc{a})$
by (\cite{HanXinUnpub}, Theorem~1). What they proved was the following:  

\begin{theorem} \label{thm: HX}
The number $|D(\vc{a})|$ is the coefficient of $x_1^{a_1}\cdots x_k^{a_k}$ in the expansion of 
\[\frac{1}{(1+x_1)\cdots(1+x_k)(1-x_1-\dots-x_k)}.\]
\end{theorem}

The proof uses scalar products of symmetric functions. We give a more direct proof, with a combinatorial flavour. The proof uses the following definition, and the bijective result of Lemma~\ref{lm:MjNj}. 

\begin{definition} \label{def:D_j}
We denote by $D_j(\vc{a})$ the set of permutations in $\SymGroup{\vc{a}}$ that have no fixed points in blocks $A_1,\dots,
A_j$. Thus, $D(\vc{a})=D_k(\vc{a})$. 

Moreover, let $D^*_j(\vc{a})$ be the set of permutations in $\SymGroup{\vc{a}}$ that have no fixed points in the first $j-1$
blocks, but have a fixed point in $A_j$. 
\end{definition}

\begin{lemma} \label{lm:MjNj}
There is a bijection between $D_j(a_1,\dots,a_k)$ and \\$D_j^*(a_1,\dots,a_{j-1},a_j+1,a_{j+1},\dots,a_k)$. 
\end{lemma}

\begin{proof}
Let $\sigma=\sigma_1\dots\sigma_n$ be a permutation in
$D_j(a_1,\dots,a_k)$, and consider the block $A_j=\{p ,
p+1,\dots,q\}$. Then there is an index $r$ such that
$\sigma_p\dots\sigma_{r-1}$ are excedances, and
$\sigma_r\dots\sigma_q$ are deficiencies.  

Now $\phi_r(\sigma)$ is a permutation of $[n+1]$. It is easy to see
that $$\phi_r(\sigma) \in
\SymGroup{(a_1,\dots,a_{j-1},a_j+1,a_{j+1},\dots,a_k)}.$$ All the
fixed points of $\sigma$ are shifted one step to the right, and one
new is added in the $j$:th block, so $$\phi(\sigma) \in
D_j^*(a_1,\dots,a_{j-1},a_j+1,a_{j+1},\dots,a_k).$$  
We see that $\psi_r(\phi_r(\sigma)) = \sigma$, so the map $\sigma\mapsto \phi_r(\sigma)$ is a bijection.
\end{proof}

We now obtain a generating function for $|D(\vc{a})|$, with a purely
combinatorial proof. In fact, we even strengthen the result to give
generating functions for $|D_j(\vc{a})|$, $j=0, \dots, k$. Theorem
\ref{thm: HX} then follows by letting $j = k$. 

\begin{theorem} \label{thm:Genf}
The number $|D_j(\vc{a})|$ is the coefficient of $x_1^{a_1}\cdots x_k^{a_k}$ in the expansion of
\begin{equation} \label{eq:GenFj}
\frac{1}{(1+x_1)\cdots(1+x_j)(1-x_1-\dots-x_k)}.
\end{equation}
\end{theorem}

\begin{proof} 
Let $F_j(\vc{x})$ be the generating function for $|D_j(\vc{a})|$, so that $|D_j(a_1,\dots,a_k)|$ is the coefficient for $x_1^{a_1}\cdots x_k^{a_k}$ in $F_j(\vc{x})$. We want to show that $F_j( x)$ is given by (\ref{eq:GenFj}). 

By definition, $|D_0(\vc{a})|=|\SymGroup{\vc{a}}|$. But a permutation in $\SymGroup{\vc{a}}$ is uniquely determined by the set of $a_1$ numbers in the first block, the set of $a_2$ numbers in the second, etc. So $|D_0|$ is the multinomial coefficient $\binom{n}{a_1,a_2,\dots ,a_k}$.  This is also the coefficient of $x_1^{a_1}\cdots x_k^{a_k}$ in the expansion of $1+(\sum x_i)+(\sum x_i)^2+\cdots$, since any such term must come from the $\left(\sum x_i\right)^n$-term. Thus,
\[F_0(\vc{x})=1+\left(\sum x_i\right)+\left(\sum x_i\right)^2+\cdots=\frac{1}{(1-x_1-\dots-x_k)}.\] 

Note that for any $j$, $D_{j-1}(\vc{a})=D_j(\vc{a})\cup D_j^*(\vc{a})$, and the two latter sets are disjoint. Indeed, a
permutation in $D_{j-1}$ either does or does not have a fixed point in the $j$:th block. Hence by Lemma \ref{lm:MjNj}, we have the identity \[|D_{j-1}(\vc{a})|=|D_j(\vc{a})|+|D_j(a_1,\dots,a_{j-1},a_j-1,a_{j+1},\dots,a_k)|.\] 
This holds also if $a_j=0$, if the last term is interpreted as $0$ in that case. 

In terms of generating functions, this gives the recursion $F_{j-1}(\vc{x})=(1+x_j)F_j(\vc{x})$. Hence
$F_0(\vc{x})=F_j(\vc{x})\prod_{i\leq j} (1+x_i)$. Thus, 
\[
F_j(\vc{x})=\frac{F_0(\vc{x})}{(1+x_1)\cdots(1+x_j)}=\frac{1+(\sum
  x_i)+(\sum x_i)^2+\cdots}{(1+x_1)\cdots(1+x_j)},
\]
and $|D_j(\vc{a})|$ is the coefficient for $x_1^{a_1}\cdots x_k^{a_k}$ in
the expansion of $F_j$. 
\end{proof}

\begin{corollary}
The number of derangements in $\SymGroup{\vc{a}}$ is the coefficient of $x_1^{a_1}\cdots x_k^{a_k}$ in the expansion of
\begin{equation}\label{eq:GenFk} F(\vc{x})=\frac{1}{(1+x_1)\cdots(1+x_k)(1-x_1-\dots-x_k)}.
\end{equation}
\end{corollary}

\begin{proof}
The set of derangements in $\SymGroup{\vc{a}}$ is just $D(\vc{a})=D_k(\vc{a})$. Letting $j=k$ in Theorem~\ref{thm:Genf} gives the generating function for $|D(\vc{a})|$.
\end{proof}

\section{An explicit enumeration} \label{sc:clofor}
It is not hard to explicitly calculate the numbers $|D(\vc{a})|$ from
here. We will use $\vc{x}^{\vc{a}}$ as shorthand for $\prod_i
x_i^{a_i}$. 

Every term $\vc{x}^{\vc{a}}$ in the expansion of $F(\vc{x})$ is obtained
by choosing $x_i^{b_i}$ from the factor $$\frac{1}{1+x_i}=\sum_{j\geq 0} (-x_i)^j,$$ for some $0 \leq b_i\leq a_i$. This gives us a coefficient of $(-1)^{\sum b_i}$. For each choice of $b_1,\dots, b_k$ we should multiply by $\vc{x}^{{\vc{a}-\vc{b}}}$ from the factor \[\frac{1}{(1-x_1-\dots-x_k)}=1+\left(\sum x_i\right)+\left(\sum x_i\right)^2+\cdots.\]

But every occurence of $\vc{x}^{{\vc{a}-\vc{b}}}$ in this expression comes from the term $(\sum x_i) ^{n-\sum b_j}$. Thus the coefficient of $\vc{x}^{{\vc{a}-\vc{b}}}$ is the multinomial coefficient \[\binom{n-\sum b_j}{a_1-b_1,\dots ,a_k-b_k}=\frac{(n-\sum b_j)!}{(a_1-b_1)!\cdots(a_k-b_k)!}.\]

Now since $|D(\vc{a})|$ is the coefficient of $\vc{x}^{\vc{a}}$ in $F_k(\vc{x})$, we conclude that
\[
\begin{split} \label{eq:Factorial}
|D(\vc{a})|&=\sum_{ \vc{0} \leq \vc{b}\leq \vc{a}}(-1)^{\sum{b_j}}\frac{(n-\sum b_j)!}{(a_1-b_1)!\cdots(a_k-b_k)!}\\ 
&=\frac{1}{\prod_i a_i!}\sum_{ \vc{0} \leq \vc{b} \leq \vc{a}}(-1)^{\sum{b_j}}\left(n-\sum b_j\right)!\prod_i\binom{a_i}{b_i}b_i!.
\end{split}
\]

While the expression \eqref{eq:Factorial} seems a bit more involved than necessary, it turns out to generalise in a nice way.

\section{Fixed point coloured permutations}

A \emph{fixed point coloured permutation} in $\lambda$ colours, or a fixed point $\lambda$-coloured permutation, is a permutation where we require each fixed point to take one of $\lambda$ colours. More formally it is a pair $(\pi, C)$ with $\pi\in\Sn$ and $C:F_\pi\rightarrow [\lambda]$, where $F_\pi$ is the set of fixed points of $\pi$. When there can be no confusion, we denote the coloured permutation $(\pi, C)$ by $\pi$. Thus, fixed point 1-coloured permutatations are simply ordinary permutations and fixed point 0-coloured permutations are derangements. The set of fixed point $\lambda$-coloured
permutations on $n$ elements is denoted $\Sn^\lambda$. 

For the number of $\lambda$-fixed point coloured permutations on $n$
elements, we use the notation $|\Sn^\lambda| = f_\lambda(n)$, the
{\em $\lambda$-factorial} of $n$. Of course, we have
$f_0(n) = D_n$ and $f_1(n) = n!$. Clearly,
\[
f_\lambda(n) = \sum_{\pi \in \Sn} \lambda^{\fix(\pi)},
\]
where $\fix(\pi)$ is the number of fixed points in $\pi$, and we use this
formula as the definition of $f_\lambda(n)$ for $\lambda \not\in
\mathbb{N}$. 

\begin{lemma} \label{lm:lamfakid}
For $\nu, \lambda \in \mathbb{C}$ and $n\in\mathbb N$, we have
\[
f_\nu(n) = \sum_j \binom{n}{j} f_{\lambda}(n - j) \cdot(\nu - \lambda)^j.
\]
\end{lemma}

\begin{proof}
It suffices to show this for $\nu, \lambda, n \in \mathbb{N}$, since
the identity is polynomial in $\nu$ and $\lambda$, so if it holds on
$\mathbb{N}\times\mathbb{N}$ it must hold on all of
$\mathbb{C}\times\mathbb{C}$. 

We divide the proof into three parts. First, assume $\nu =
\lambda$. Then all terms in the sum vanish except for $j = 0$, when we
get $f_\nu(n) = f_\lambda(n)$.

Secondly, assuming $\nu > \lambda$, we let $j$ denote the number of fixed points
in $\pi \in \Sn^\nu$ which are coloured with colours from
$[\lambda + 1, \nu]$. These fixed points can be chosen in $\binom{n}{j}$ ways,
there are $f_\lambda(n - j)$ ways to permute and colour the
remaning elements, and the colours of the high coloured fixed points
can be chosen in $(\nu - \lambda)^j$ ways. Thus, the equality holds.

Finally, assuming $\nu < \lambda$, we prescribe $j$ fixed
points in $\pi \in \Sn^\lambda$ which only get to choose their colours
from $[\nu + 1, \lambda]$. These fixed points can be chosen in
$\binom{n}{j}$ ways, the remaining elements can be permuted in $f_\lambda(n - j) $ ways and the chosen fixed points can be coloured in
$(\lambda - \nu)^j$ ways, so by the principle of inclusion-exclusion,
the equality holds. 
\end{proof}

With $\lambda = 1$ and replacing $\nu$ by $\lambda$, we find that
\begin{equation} \label{eq: lamfaktrunc}
f_\lambda(n) = n! \left( 1 + \frac{(\lambda - 1)}{1!} + \frac{(\lambda - 1)^2}{2!} + 
\cdots + \frac{(\lambda - 1)^n}{n!} \right) = n! \ \expn(\lambda - 1).
\end{equation}
Here we use $\exp_n$ to denote the truncated series expansion of the exponential function.
In fact, $\lim_{n \rightarrow \infty} n! \re^{(\lambda - 1)} - f_\lambda(n) = 0$ for
all $\lambda \in \mathbb{C}$, although we cannot in general approximate $f_\lambda(n)$ by the nearest integer of $n! \re^{\lambda - 1}$ as for derangements.

The formula \eqref{eq: lamfaktrunc} also shows that 
\begin{equation} \label{lamfakrec}
f_\lambda(n) = n f_\lambda(n-1) + (\lambda - 1)^n, \qquad f_\lambda(0) = 1
\end{equation}
which generalises the well known recursions $|D_n| = n |D_{n-1}| + (-1)^n$ and $n! = n (n-1)!$.

\section{Enumerating $D(\vc{a})$ using fixed point coloured permutations}

Another consequence of \eqref{eq: lamfaktrunc} is that the $\lambda$-factorial satisfies the following rule for differentiation, which is similar to the rule for differentiating powers of $\lambda$:
\begin{equation} \label{ruleofdifferentiation} \frac{d}{d\lambda}f_\lambda\left(n\right) = n\cdot f_\lambda\left(n-1\right).\end{equation}

This follows immediately from the fact that $f_\lambda(n)$ is equal to $n!$
times the truncated series expansion of $\re^{\lambda-1}$, as in
(\ref{eq: lamfaktrunc}). Regarding
$n$ as the cardinality of a set $X$, the differentiation rule
(\ref{ruleofdifferentiation}) translates to 
\begin{equation} \label{setdifferentiation}
\frac{d}{d\lambda}f_\lambda\left(\left|X\right|\right) = \sum_{x\in X}f_\lambda\left(\left|X\smallsetminus\{x\}\right|\right).\end{equation}

Products of $\lambda$-factorials can of course be differentiated by the product formula. This implies that if $X_1,\dots X_k$ are disjoint sets, then $$\frac{d}{d\lambda}\prod_i f_\lambda\left(\left|X_i\right|\right)= \sum_{x\in \cup X_j} \prod_i f_\lambda\left(\left|X_i\smallsetminus\{x\}\right|\right).$$

Now consider the expression 

\begin{equation}\label{rty}
\sum_{B\subseteq [n]}(-1)^{\left|B\right|}f_\lambda\left(\left|[n]\smallsetminus B\right|\right)\prod_i^k f_\lambda\left(\left|A_i\cap B\right|\right).
\end{equation} 

This is obtained from \eqref{eq:Factorial} by deleting the factor $1/\prod_i a_i!$ and replacing the other factorials by $\lambda$-factorials. For $\lambda = 1$, \eqref{rty} is therefore $\left|\Phi_{\vc{a}}^{-1}(D(\vc{a}))\right|$, the number of permutations that, when sorted in decreasing order within the blocks, have no fixed points. We want to show that \eqref{rty} is independent of $\lambda$. The derivative of \eqref{rty} is, by the rule \eqref{setdifferentiation} of differentiation, 

\begin{equation} \label{derivative}
\sum_{B\subseteq [n]}(-1)^{\left|B\right|} \sum_{x=1}^n f_\lambda\left(\left|[n]\smallsetminus B\smallsetminus \{x\}\right|\right)\prod_{i=1}^k f_\lambda\left(\left|(A_i\cap B)\smallsetminus \{x\}\right|\right).
\end{equation}

Here each product of $\lambda$-factorials occurs once with $x\in B$ and once with $x\notin B$. Because of the sign $(-1)^{\left|B\right|}$, these terms cancel. Therefore \eqref{derivative} is identically zero, which means that \eqref{rty} is independent of $\lambda$. Hence we have proven the following theorem:

\begin{theorem} \label{thm:lamfak}
For any $\lambda\in \mathbb C$, the identity 
\begin{equation} \label{eq:lamfak}
\left|\Phi_{\vc{a}}^{-1}(D(\vc{a}))\right|=\sum_{{0}
\leq\vc{b}\leq\vc{a}}(-1)^{\sum{b_j}}\cdot f_\lambda\left(n-\sum
b_j\right) \prod_i\binom{a_i}{b_i}\cdot f_\lambda\left(b_i\right)
\end{equation}
holds.
\end{theorem}

A particularly interesting special case is when we put $\lambda=0$. In this case, $f_0(n) = D_n$, so

\begin{equation} \label{lambda=0}
|D(\vc{a})| = \frac{1}{\prod_i a_i!}\sum_{\vc{0} \leq \vc{b} \leq \vc{a}}(-1)^{\sum{b_j}}D_{n-\sum b_j}\prod_i\binom{a_i}{b_i}D_{b_i}.
\end{equation}

This equation has some advantages over \eqref{eq:Factorial}. It has a clear main term, the one with $\vc{b} = \vc{0}$. Moreover, since $D_1 = 0$, the number of terms does not increase if blocks of length 1 are added.
 
\section{A recursive proof of Theorem \ref{thm:lamfak}} \label{sc: lamrec}

We will now proceed by proving Theorem \ref{thm:lamfak} in a more
explicit way. This proof will use the sorting operator
$\Phi_{\vc{a}}$ and our notion of fixed point coloured permutations,
and will not need to assume the case $\lambda=0$ to be known. First we
need some new terminology.

\begin{definition}
We let $\hat D_j(\vc{a})\subseteq \SymGroup{\vc{a}}$ denote the set of permutations in $\SymGroup{\vc{a}}$ that have a fixed point in $A_j$, but that have no fixed points in any other block. 
\end{definition}

The proof of Lemma \ref{lm:MjNj} goes through basically unchanged,
when we allow no fixed points in $A_{j+1},\dots, A_k$:

\begin{lemma} \label{lm:FixRem}
There is a bijection between $D(a_1,\dots,a_k)$ and 
\\$\hat D_\ell(a_1,\dots,a_{\ell-1},a_\ell+1,a_{\ell+1},\dots,a_k)$.
\end{lemma}

We now have the machinery needed to give a second proof of Theorem \ref{thm:lamfak}.

\begin{proof}[Proof of Theorem \ref{thm:lamfak}]
It suffices to show this for $\lambda=1,2,\dots$, because then for
given $\vc{a}$, the expression is just a polynomial in $\lambda$,
which is constant on the positive integers, and hence constant. So
assume $\lambda$ is a positive integer.  

In the case where $\vc{a}=(1,\dots ,1)$, $\Phi$ is the identity, and
(\ref{eq:lamfak}) can be written \[\left|D(\vc{a})\right|=\sum
(-1)^j\binom{n}{j} f_\lambda\left(n-j\right)\cdot \lambda^j ,\] where we have made the
substitution $j=\sum b_i$. This is true by letting $\nu = 0$ in Lemma
\ref{lm:lamfakid}. We will proceed by induction to show that
(\ref{eq:lamfak}) holds for any composition $\vc{a}$.  

Suppose it holds for the compositions
$\vc{a'}=(a_1,\dots,a_{\ell-1}, 1,\dots, 1, a_{\ell+1},\dots, a_k)$ (with
$a_\ell$ ones in the middle) and $\vc{a''}=(a_1,\dots,a_{\ell-1},
a_\ell-1, a_{\ell+1},\dots, a_k)$. We will prove that it holds for
$\vc{a}=(a_1,\dots,a_k)$. 

First, we enumerate the disjoint union
$\Phi_{\vc{a}}^{-1}(D(\vc{a}))\cup\Phi_{\vc{a}}^{-1}(\hat
D_\ell(\vc{a}))$. This is just the set of permutations that, when
sorted, have no fixed points except possibly in $A_\ell$. 

Sort these decreasingly in all blocks except $A_\ell$ (which means
that we apply $\Phi_{\vc{a^{'}}}$ to them). Then we enumerate them
according to the number $t$ of fixed points in $A_\ell$. Note that
$A_\ell$ splits into several blocks $A_\alpha$, one for each non-fixed
point. We let $p$ denote the sum
of the $b_\alpha$:s for these blocks, which gives $\prod f_\lambda(b_\alpha)=\lambda^{p}$. 

If we let $\vc b$ range over $k$-tuples $(b_1 , \dots , b_k)$ and $\hat{\vc{b}}_\ell$ range over the $(k-1)$-tuples $(b_1,\dots, b_{\ell-1},b_{\ell+1},\dots ,b_k)$, we use the induction hypothesis to get 

\[\begin{split}
& \quad |\Phi_{\vc{a}}^{-1}(D(\vc{a}))|+|\Phi_{\vc{a}}^{-1}(\hat
 D_\ell(\vc{a}))| \\
&=\sum_{\hat{\vc{b}}_\ell} (-1)^{\sum_{i \neq \ell} b_i} \left( \prod_{i\neq
 \ell}\binom{a_i}{b_i} f_\lambda(b_i) \right) \sum_t \binom{a_\ell}{t}
 \sum_p \binom{a_\ell-t}{p} (-1)^p \cdot f_\lambda\big(n - t - \sum_{i \neq \ell} b_i - p\big) \lambda^p \\
&= \sum_{\hat{\vc{b}}_\ell} (-1)^{\sum_{i \neq \ell} b_i} \left( \prod_{i\neq
 \ell}\binom{a_i}{b_i}f_\lambda(b_i) \right) \sum_{b_\ell} \sum_p f_\lambda\big(n -
 \sum_{i \neq \ell} b_i - b_\ell\big) (-1)^p \lambda^p \binom{a_\ell}{b_\ell-p}
 \binom{a_\ell-b_\ell+p}{p} \\
&= \sum_{\hat{\vc{b}}_\ell} (-1)^{\sum_{i \neq \ell} b_i} \left( \prod_{i\neq
 \ell}\binom{a_i}{b_i}f_\lambda(b_i) \right) \sum_{b_\ell} \sum_p f_\lambda\big(n -
 \sum_{i \neq \ell} b_i - b_\ell\big) (-1)^p \lambda^p \binom{a_\ell}{b_\ell}
 \binom{b_\ell}{p} \\
&= \sum_{\hat{\vc{b}}_\ell} (-1)^{\sum_{i \neq \ell} b_i} \left( \prod_{i\neq
 \ell}\binom{a_i}{b_i}f_\lambda(b_i) \right) \sum_{b_\ell} (-1)^{b_\ell} f_\lambda\big(n -
 \sum_{i \neq \ell} b_i - b_\ell\big) \binom{a_\ell}{b_\ell}
 (\lambda - 1)^{b_\ell} \\
&= \sum_{\vc{b}} (-1)^{\sum b_i} \left( \prod_{i\neq
 \ell}\binom{a_i}{b_i}f_\lambda(b_i) \right) f_\lambda \big(n -
 \sum b_i\big) \binom{a_\ell}{b_\ell} (\lambda - 1)^{b_\ell}.
\end{split}
\]

This expression makes sense, as the binomial coefficients become zero
unless ${0} \leq \vc{b} \leq \vc{a}$. On the other hand, by Lemma
\ref{lm:FixRem} and the induction hypothesis,  
\[
\begin{split}
|\Phi_{\vc{a}}^{-1}(\hat D_\ell(\vc{a}))| &=\prod {a_i!} \cdot |\hat D_\ell(\vc{a})|=\prod a_i!\cdot |D(\vc{a''})|= a_\ell\cdot |\Phi_{{a''}}^{-1}(D(\vc{a''}))| \\
&=a_\ell\cdot \sum_{\vc{b}}(-1)^{\sum b_i}\binom{a_\ell-1}{b_\ell}f_\lambda(b_\ell) \cdot f_\lambda\left(n-1-\sum b_i\right)\prod_{i\neq l}\binom{a_i}{b_i}f_\lambda(b_i).
\end{split}
\]
Noting that $a_\ell\binom{a_\ell-1}{b_\ell}=(b_\ell+1)\binom{a_\ell}{b_\ell+1}$, we shift the parameter $b_\ell$ by one, and get
\[|\Phi_{\vc{a}}^{-1}(\hat D_\ell(\vc{a}))|=-\sum_{\vc{b}}(-1)^{\sum b_i}\binom{a_\ell}{b_\ell}b_\ell\cdot f_\lambda(b_\ell-1)\cdot f_\lambda\left(n-\sum b_i\right)\prod_{i\neq l}\binom{a_i}{b_i}f_\lambda(b_i) .
\]
Thus we can write
\[
\begin{split}
|\Phi_{\vc{a}}^{-1}(D(\vc{a}))|&=\sum_{\vc{b}}(-1)^{\sum
 b_i}\binom{a_\ell}{b_\ell}(\lambda-1)^{b_\ell}f_\lambda \big(n-\sum b_i\big)\prod_{i\neq
 l}\binom{a_i}{b_i}f_\lambda (b_i) - |\Phi_{\vc{a}}^{-1}(\hat
 D_\ell(\vc{a}))| \\
&=\sum_{\vc{b}} (-1)^{\sum b_i}f_\lambda\big(n-\sum
 b_i\big)\binom{a_\ell}{b_\ell}\left((\lambda - 1)^{b_\ell} +b_\ell\cdot
 f_\lambda\left(b_\ell-1\right)\right)\prod_{i\neq l}\binom{a_i}{b_i}f_\lambda(b_i) \\
&= \sum_{\vc{b}} (-1)^{\sum b_i} f_\lambda\big(n - \sum b_i\big) \prod_i \left(
 \binom{a_i}{b_i} f_\lambda(b_i) \right).
\end{split}
\]

\end{proof} 

We also note that Theorem \ref{thm:lamfak} can be used to enumerate
permutations in $\SymGroup{\vc{a}}$ with $\mu$ allowed fixed point
colours, and even $\mu_i$ fixed point colours in block $A_i$.

\begin{corollary}
For any $\lambda \in \mathbb{C}$ and natural numbers $\mu_i, 1 \leq i
\leq k$, the number of permutations $(\pi, C)$ where $\pi \in
\SymGroup{\vc{a}}$ and $(j \in A_i, \pi(j) = j) \Rightarrow C(j) =
[\mu_i]$ is given by
\[
\sum_{\vc{0} \leq \vc{c} \leq \vc{1}} \sum_{\vc{0} \leq \vc{b} \leq
\vc{a} - \vc{c}} (-1)^{\sum{b_j}} f_\lambda\left(\sum a_j - \sum c_j -\sum
b_j\right) \prod_i \binom{a_i-c_i}{b_i} f_\lambda(b_i)\cdot
\mu_i^{c_i}. 
\] 
\end{corollary}

\begin{proof}
The numbers $c_i$ are one if $A_i$ contains a fixed point and zero
otherwise. We may remove these fixed points and consider a fixed point
free permutation, enumerated above. We then reinsert the fixed points
and colour them in every allowed combination.
\end{proof}

 
\section{A correlation result} \label{sec:Correlation}
 
Taking a permutation at random in $\Sn$, the chances are about $1/\re$ that it
is fixed point free, since there are $n!\expn(-1)$ fixed point free permutations
in $\Sn$. Moreover, there are $n!/\vc{a}!$ 
permutations in $\SymGroup{\vc{a}}$.

If belonging to $\SymGroup{\vc{a}}$ and being fixed point free were two
independent events, we would have $n!\expn(-1)$ permutations in
$\Phi^{-1}(D(\vc{a}))$. This is not the case, although the leading term in
(\ref{eq:Factorial}) is this very number. The following theorem implies, in particular, that belonging to $\SymGroup{\vc{a}}$ and being fixed point free are almost always positively correlated events. The sole exception is when $\vc{a}$ is a single block of odd length, in which
case every permutation gets a fixed point when sorted. 

For two compositions $\vc{a}$ and $\vc{b}$ of $n$, we say that $\vc a \geq \vc b$
if, when sorted decreasingly, $\sum_{i \leq j} a_i \geq \sum_{i \leq j}b_i$ for all $j$.

\begin{theorem} \label{thm:correl} If $\vc{a} \geq \vc{b}$ and $\vc{a}$ is not a single block of odd size, then
\[
 |\Phi^{-1}_{\vc{a}}(D(\vc{a}))| \geq |\Phi^{-1}_{\vc{b}}(D(\vc{b}))|.
\] 
\end{theorem}

Theorem~\ref{thm:correl} implies that, with only the trivial exception, the proportion of derangements among permutations with descending blocks $\vc{a}$ is minimal when $\vc{a} = (1,1,\dots,1)$, that is, when there are no prescribed descents.

The theorem will follow from a series of lemmata. The main point is proving that shifting any position from a smaller block
to a larger one almost never decreases the number $|\Phi^{-1}_{\vc{a}}(D(\vc{a}))|$. Equivalently, for fixed $a_3,\dots a_k$, and $a=a_1+a_2$ fixed, the function $|\Phi^{-1}(D(\vc{a}))|$ is unimodal in $a_1$ (with the trivial exception).

Let $F(a_1, a_2, s)$ be the number
of linear orders of the union of $a - s$ elements in $[a_1 + a_2]$ (regardless
which) and $[a+1, a+s]$, such that if these are sorted decreasingly in 
$[a_1, a_2]$, there is no fixed point. For instance, $F(3, 0, 1) = 2\cdot 3!$, 
counting all ways to scramble $431$ and $432$, since $421$ has a fixed point. 

The reason for counting these orders is that given $s$ elements from $[a+1, n]$
and $a - s$ elements from $[a]$ in blocks $a_1$ and $a_2$, the number
of ways to put the remaining elements in the remaining positions does not
depend on $a_1$ and $a_2$, but only on $a$. 

We also define the function $G(a_1, a_2, s)$ as the sum over $m$ of the number
of permutations in $\SymGroup{(a_1-m, a_2-(s-m), m, s-m)}$ such that
there are no fixed points in the first two blocks. The relation to 
$F(a_1, a_2, s)$ is the following.

\begin{lemma}
We have
\[
F(a_1, a_2, s) = a_1! a_2! G(a_1, a_2, s).
\]
\end{lemma}

\begin{proof}
$F(a_1, a_2, s)$ gives $a_1! a_2!$ orders for every way to sort the elements
in two decreasing blocks of lengths $a_1$ and $a_2$. But then the initial
$m$ elements in the block $a_1$ will be larger than $a$ and can hence not
produce a fixed point, as for the first $s - m$ elements of the block $a_2$.
Thus, we could equally well take four decreasing sequences from $[a]$ of lengths
$(m, a_1-m, s-m, a_2-(s-m))$, not bothering about fixed points in the first
and third block. The statement follows by rearranging the blocks.
\end{proof}

\begin{lemma}\label{lm:Grek}
We have
\[
G(a_1, a_2, s) = \sum_{b_1, b_2\geq 0} (-1)^{b_1+b_2} \binom{a_1 + a_2 - b_1 - b_2}{s}
\binom{a_1 + a_2 - b_1 - b_2}{a_1 - b_1}.
\]
\end{lemma}

\begin{proof}
Using Theorem \ref{thm:Genf} with $j = 2$, we immediately get
\[
\begin{split}
G(a_1, a_2, s) &= \sum_{m, b_1, b_2} (-1)^{b_1 + b_2}\binom{a_1 + a_2 - b_1 - b_2}
{m, s-m, a_1 - m - b_1, a_2 - (s - m) - b_2} \\
&= \sum_{b_1, b_2} (-1)^{b_1 + b_2} \binom{a_1 + a_2 - b_1 - b_2}{s}
\sum_m \binom{s}{m} \binom{a_1 + a_2 - s - b_1 - b_2}{a_1 - m - b_1} \\
&= \sum_{b_1, b_2} (-1)^{b_1+b_2} \binom{a_1 + a_2 - b_1 - b_2}{s}
\binom{a_1 + a_2 - b_1 - b_2}{a_1 - b_1},
\end{split}
\] where all sums are taken over the nonnegative integers.
\end{proof}

We now wish to establish a recurrence, which by induction will show that the sequence $F(a_1,a-a_1,s)$ is unimodal with respect to $a_1$. We start by computing neccessary base cases.

\begin{lemma}
If $a_1$ is odd, we have $G(a_1, 0, 0)=0$ and if $a_1$ is even, $G(a_1, 0, 0)=1$. For all $a_1$ we have $G(a_1, 0, a_1) = 1$.
Moreover, $G(a_1, 1, 0) = a_1/2$ for even $a_1$ and 
$G(a_1, 2, 0) = ((a_1 + 1)/2)^2$ for odd $a_1$.
\end{lemma}

\begin{proof}
The first assertion follows from the fact that permutations without 
ascents have a fixed point if and only they have an odd number of elements.
It is also clear that $G(a_1, 0, a_1) = 1$, since we allow fixed points
in the last two parts.


A permutation $\pi\in\SymGroup{(a_1,1)}$ is determined by its last
element, and if $a_1$ is even, it is easy to see that $\pi$ is 
fixed point free iff $a_1/2<\pi_{a_1+1}\leq a_1$. Hence we get $G(a_1, 1,
0) = a_1/2$. 

Further, if $a_1$ is odd we get
\[
\begin{split}
G(a_1, 2, 0) &= \frac{1}{2} \sum_{b_1 = 0}^{a_1} (-1)^{b_1} 
\left( (a_1 - b_1)^2 + (a_1 - b_1) + 2 \right) \\
&= \frac{a_1^2 + a_1 + 2}{2} + \frac{1}{2} \sum_{b_1 = 1}^{a_1} (-1)^{b_1} 
\left( (a_1 - b_1)^2 + (a_1 - b_1) + 2 \right) \\
&= \frac{a_1^2 + a_1 + 2}{2} - G(a_1-1, 2, 0),
\end{split}
\]
which gives, by induction and $G(1, 2, 0) = 1 = (2/2)^2$,
\[
\begin{split}
G(a_1, 2, 0) &= \frac{a_1^2 + a_1 + 2}{2} - \frac{(a_1 - 1)^2 + (a_1 -
  1) + 2}{2} + G(a_1-2, 2, 0) \\ 
&= a_1 + G(a_1 - 2, 2, 0) = a_1 + \left( \frac{a_1 - 1}{2} \right)^2
= \left( \frac{a_1 + 1}{2} \right)^2.
\end{split} 
\]
\end{proof}

The case $s = 0$ has to be treated separately. We start with proving unimodality
for this case, in the two following lemmas. 

\begin{lemma}\label{lm:s=0}
For $a_1, a_2 \geq 1$ and $s = 0$ we have 
\[
G(a_1, a_2, 0) = G(a_1-1, a_2, 0) + G(a_1, a_2-1, 0) + (-1)^{a_1+a_2}.
\]
\end{lemma}

\begin{proof}
When $s=0$, Lemma \ref{lm:Grek} reduces to 
\[
G(a_1, a_2, 0) = \sum_{b_1, b_2\geq 0} (-1)^{b_1+b_2}
\binom{a_1 + a_2 - b_1 - b_2}{a_1 - b_1},
\]
which gives
\[
\begin{split}
& G(a_1-1, a_2, 0) + G(a_1, a_2-1, 0) \\
&= \sum_{b_1, b_2\geq 0} (-1)^{b_1+b_2}
\left[\binom{a_1 - 1 + a_2 - b_1 - b_2}{a_1 - 1 - b_1} + 
\binom{a_1 + a_2 - 1 - b_1 - b_2}{a_1 - b_1}\right] \\
&= \sum_{b_1, b_2\geq 0} (-1)^{b_1+b_2}
\binom{a_1 + a_2 - b_1 - b_2}{a_1 - b_1} - (-1)^{a_1+a_2}.
\end{split}
\]
The last term corresponds to $b_1=a_1$, $b_2=a_2$, which gives a term in the second sum but not in the first one.
\end{proof}

\begin{lemma}
For $a_1 \geq a_2 \geq 1$ and $s = 0$ we have
\[
F(a_1+1, a_2-1, 0) \geq F(a_1, a_2, 0),
\]
unless $a_1$ is even and $a_2 = 1$.
\end{lemma}

\begin{proof}
We wish to show that $F(a_1+1, a_2-1, 0) - F(a_1, a_2, 0) \geq 0$. For 
$a_2 = 1$, this is clearly true for odd $a_1$, but not for even $a_1$.
More generally, we get by Lemma \ref{lm:s=0} that 
\begin{equation} \label{theEquationAbove}
\begin{split}
F(a_1+1, a_2-1, 0) - F(a_1, a_2, 0) 
&= (a_1+1) F(a_1, a_2-1, 0) + (a_2-1) F(a_1+1, a_2-2, 0) \\
&\qquad - a_1 F(a_1-1, a_2, 0) - a_2 F(a_1, a_2-1, 0) \\
&= a_1 (F(a_1, a_2-1, s) - F(a_1-1, a_2, s)) \\
&\qquad + (a_2 - 1) (F(a_1+1, a_2-2, s) - F(a_1, a_2-1, s)),  
\end{split}
\end{equation}
which is non-negative by induction for $a_2 \geq 3$, and for $a_2 = 2$
with odd $a_1$. Thus, what remains is the case $a_2 = 2$ with even $a_1$.
Equation \eqref{theEquationAbove} then specialises to
\[
\begin{split}
F(a_1+1, 1, 0) - F(a_1, 2, 0) &= (a_1 - 1) F(a_1, 1, 0) - a_1 F(a_1-1, 2, 0) \\
&= (a_1-1) a_1! \frac{a_1}{2} - a_1 (a_1-1)! \left(\frac{a_1}{2} \right)^2 \\
&= a_1! \frac{a_1}{4} (2a_1 - 2 - a_1) = a_1! \frac{a_1}{4} (a_1 - 2),
\end{split}
\]
which is non-negative for $a_1 \geq 2$.
\end{proof}

We can now proceed with the case $s\geq 1$.

\begin{lemma}
For $a_1 \geq 1$, $a_2 \geq 0$ and $s \geq 1$ we have
\[
G(a_1, a_2, s) = G(a_1-1, a_2, s) + G(a_1-1, a_2, s-1) + G(a_1, a_2-1, s)
+ G(a_1, a_2-1, s-1).
\]
\end{lemma}

\begin{proof}

By Lemma \ref{lm:Grek}, and writing $c_i:=a_i-b_i$, we get
\[
\begin{split}
& G(a_1-1, a_2, s) + G(a_1-1, a_2, s-1) + G(a_1, a_2-1, s)
+ G(a_1, a_2-1, s-1)\\
&= \sum_{\vc{0}\leq\vc{b}} (-1)^{b_1+b_2}\left[\binom{c_1+c_2-1}{s} + \binom{c_1+c_2-1}{s-1}\right]\left[\binom{c_1+c_2-1}{c_1-1}+\binom{c_1+c_2-1}{c_1}\right] \\
&= \sum_{\vc{0}\leq\vc{b}}(-1)^{b_1+b_2}\binom{c_1+c_2}{s}\binom{c_1+c_2}{c_1} = G(a_1, a_2, s).
\end{split}
\]
\end{proof}

\begin{lemma}
For $a_1 \geq a_2 \geq 0$ and $s \geq 1$ we have
\[
F(a_1+1, a_2-1, s) \geq F(a_1, a_2, s).
\]
\end{lemma}

\begin{proof}
The previous lemma translates to 
\[
F(a_1, a_2, s) = a_1(F(a_1-1, a_2, s) + F(a_1-1, a_2, s-1)) + a_2(F(a_1, a_2-1, s)
+ F(a_1, a_2-1, s-1)).
\]
Thus, with $H(a_1, a_2, s) = F(a_1, a_2, s) + F(a_1, a_2, s-1)$ for shorthand,
we get 
\[
\begin{split}
F(a_1+1, a_2-1, s) - F(a_1, a_2, s) &= 
(a_1+1) H(a_1, a_2-1, s) + (a_2 - 1) H(a_1+1, a_2-2, s) \\
&\qquad - a_1 H(a_1-1, a_2, s) - a_2 H(a_1, a_2-1, s) \\
&= a_1 (H(a_1, a_2-1, s) - H(a_1-1, a_2, s)) \\
&\qquad + (a_2 - 1) (H(a_1+1, a_2-2, s) - H(a_1, a_2-1, s)),
\end{split}
\]
which is non-negative by induction. 
\end{proof}

\begin{proof} [Proof of Theorem~\ref{thm:correl}]
Ignoring the case with only one block of odd length, fix a choice of $s$ elements from $[a+1,n]$ to be placed in the first two blocks. We have shown that more of these permutations become derangements when sorted in $(a_1+1, a_2-1, \vc{a'})$, than when sorted in $(a_1,a_2,\vc{a'})$. Summing over all such choices of $s$ elements, we get that  \[
|\Phi^{-1}_{\vc{(a_1+1,a_2-1,\vc{a'})}}(D(a_1+1,a_2-1,\vc{a'}))| \geq |\Phi^{-1}_{\vc{(a_1,a_2,\vc{a'})}}(D(a_1,a_2,\vc{a'}))|,
\] 
when $a_1\geq a_2\geq 1$, unless $({a_1,a_2,\vc{a'})} = (2m,1,\vc{0})$.

Since $|D(\vc{a})|$ is invariant under reordering the blocks, it follows that $|\Phi^{-1}_{\vc{a}} D(\vc{a})|$ increases when moving positions from smaller to larger blocks. This completes the proof of Theorem \ref{thm:correl}. 
\end{proof}



\section{Euler's difference tables fixed point coloured} \label{sec: Euler}

Leonard Euler introduced the integer table $(e^k_n)_{0 \leq k \leq n}$
by defining $e^n_n = n!$ and $e^{k-1}_n = e^k_n - e^{k-1}_{n-1}$ for
$1 \leq k \leq n$. Apparently, he never gave a combinatorial interpretation,
but a simple one is this: $e^k_n$ gives the number of permutations $\pi \in \Sn$
such that there are no fixed points on the last $n-k$ positions. Thus, 
$e^0_n = D_n$. 

It is clear from the recurrence that $k!$ divides $e^k_n$. Thus, we can
define the integers $d^k_n = e^k_n/k!$. These have recently been studied
by Fanja Rakotondrajao \cite{Rak2007},
and the combinatorial interpretation of $d^k_n$ given there was that
they count the number of permutations $\pi \in \Sn$ such that there are no
fixed points on the last $n-k$ positions and such that the first $k$ elements
are all in different cycles.

We will now generalise these integer tables to any number $\lambda$ of
fixed point colours, give a combinatorial interpretation that is
more in line with the context of this article, and bijectively prove
the generalised versions of the relations in \cite{Rak2007}.

Let $e^k_n(\lambda)$ be defined by $e^n_n(\lambda) = n!$ and
$e^{k-1}_n(\lambda) = e^k_n(\lambda) + (\lambda - 1)
e^{k-1}_{n-1}(\lambda)$. Then, a natural combinatorial interpretation for
non-negative integer $\lambda$ is that $e^k_n(\lambda)$ count the
number of permutations $\pi \in \Sn$ such that fixed points on the
last $n-k$ positions may be coloured in any one of $\lambda$ colours.

Similarly, we can define $d^k_n(\lambda) = e^k_n(\lambda)/k!$ and
interpret these numbers as counting the number of permutations
$\pi \in \SymGroup{(k, 1, 1, \ldots, 1)} \subseteq \Sn$ such that
fixed points on the last $n-k$ positions may be coloured in
any one of $\lambda$ colours. The set of these permutations
is denoted $D^k_n(\lambda)$. Thus, our intepretation for $\lambda = 0$
states that apart from forbidding fixed points at the end, we
also demand that the first $k$ elements are in descending order.
Equivalently, we could have considered permutations ending with $k-1$
ascents and having $\lambda$ fixed point colours in the first $n-k$
positions, to be closer to the setting in \cite{CorGesSavWil2007}.

There are a couple of relations that we can prove bijectively
with this interpretation, generalising the results with $\lambda = 0$
from \cite{Rak2007}. In our proofs, we will use the following
conventions. If $\lambda$ is a positive integer, and $(\pi,C)$ is a
permutation with a colouring of some of its fixed points, we will call
the colour $1$ the default colour. Fixed points $i$ with $C(i)>1$ will
be called \emph{essential fixed points}.  

When fixed points are deleted and inserted using the maps $\phi_{F+i}$
and $\psi_F$, they keep their colour. So for example, if $2$ is
coloured red in $\pi=321$, then so is the fixed point $3$ in
$\phi_{F+1}\circ\psi_F(\pi)=213$. Fixed points that are inserted but
not explicitly coloured, will be assumed to have the default colour. 

\begin{proposition} \label{prop:1}
For integers $1 \leq k \leq n$ and $\lambda \in \mathbb{C}$ we have
\[
d^{k-1}_n(\lambda) = k d^k_n(\lambda) + (\lambda - 1)
d^{k-1}_{n-1}(\lambda). 
\]
\end{proposition}

\begin{proof}
Assume $\lambda \geq 1$ is an integer. The left hand side counts the
elements in $D^{k-1}_n(\lambda)$ and the 
right hand side the elements in $\left([k] \times D^k_n(\lambda)\right) \cup 
\left([2, \lambda] \times D^{k-1}_{n-1}(\lambda)\right)$. We will give a bijection
$\theta: \left([k] \times D^k_n(\lambda)\right) \cup \left([2, \lambda] \times
D^{k-1}_{n-1}(\lambda)\right) \rightarrow D^{k-1}_n(\lambda)$, thereby
proving these sets to be equinumerous.  

For $(\pi, C) \in D^k_n(\lambda)$, $j\in[k]$, let $\theta(j, \pi) = \phi_{k, \pi_j}\circ
\psi_j(\pi)$, which takes out $\pi_j$ and inserts it at position $k$. For
$\pi \in D^{k-1}_{n-1}(\lambda)$, $c \in [2, \lambda]$, let $\theta(c, \pi) = \phi_k(\pi)$
and $C(k) = c$, that is we insert an essential fixed point $k$, coloured $c$. 
Now, $\theta$ is clearly invertible, and thus a bijection.
Since $d^k_n(\lambda)$ is a polynomial in $\lambda$ and the equation
holds for all integer $\lambda \geq 1$, it clearly holds for all
$\lambda \in \mathbb{C}$.
\end{proof}

\begin{proposition} \label{prop:2}
For integers $0 \leq k \leq n - 1$ and $\lambda \in \mathbb{C}$ we have
\[
d^k_n(\lambda) = n d^k_{n-1}(\lambda) + (\lambda - 1)
d^{k-1}_{n-2}(\lambda). 
\]
\end{proposition}

\begin{proof}
Assume $\lambda \geq 1$ is an integer. We seek a bijection 
\\$\eta: \left([n] \times D^k_{n-1}(\lambda)\right) \cup \left([2, \lambda] \times
D^{k-1}_{n-2}(\lambda)\right) \rightarrow D^k_n(\lambda)$, thereby
proving these sets to be equinumerous.  

For $(\pi, C) \in D^k_{n-1}(\lambda)$, $j\in[n]$, let $F_j = \{i \in F_\pi | C(i) > 1, 
i < j\}$ be the essential fixed points in $\pi$ less than $j$. Then, $\eta(j, \pi) = \phi_F\circ\phi_{k+1,
j-|F|}\circ\psi_F(\pi)$, which inserts the element $j$ as soon as possible after position $k$, without disturbing the essential fixed 
points. This accounts for all $(\pi, C) \in  D^k_n(\lambda)$ where the segment of essential fixed points starting at $k+1$ is
either empty or is followed by an element above or on the diagonal. 

To map to the rest, we take $(\pi, C) \in D^{k-1}_{n-2}(\lambda)$ and let
$F = \{j \in F_\pi | C(j) > 1\}$ be the essential fixed points in
$\pi$. Further, let $F+2 = \{f + 2| f \in F\}$ and let 
$m$ be the last element in 
$\psi_{[k+1, n]}(\pi)$, the reduced permutation of the
first $k$ elements in $\pi$. For $c\in[2, \lambda]$, we set $\eta(c, \pi) = 
\phi_{F+2}\circ\phi_{k+1}\circ\phi_{k+1, m}\circ\Phi_{(k, 1, \ldots,
  1)}\circ\psi_F(\pi)$ and $C(k+1) = c$. In words, we   
insert a fixed point $k+1$ with a non-default 
colour $c$, a smaller number $m$ at position $k+2$, and sort $\pi_k$ into the initial decreasing sequence, while
maintaining the positions of the fixed points, relative to the right border of
the permutation. The map may look anything but injective since we use
$\Phi_{(k, 1, \ldots, 1)}$, but since $m$ is deducable from $\eta(c,
\pi)$, the map really is injective. This gives all $(\pi, C) \in
D^k_n(\lambda)$ where the segment of essential fixed
points starting at $k+1$ is followed by an element below the
diagonal. Hence, we are done.
\end{proof}

\begin{proposition} \label{prop:3}
For integers $0 \leq k \leq n - 1$ and $\lambda \in \mathbb{C}$ we have
\[
d^k_n(\lambda) = \left(n + (\lambda - 1)\right) d^k_{n-1}(\lambda) - (\lambda -
1)\left(n - k - 1\right) d^k_{n-2}(\lambda). 
\]
\end{proposition}

\begin{proof}
Assume $\lambda \geq 1$ is an integer. We will give a map 
\[
\zeta_1: \left(\left(\left\{(j, 1) | j \in [n]\right\} \cup \left\{(k+1, c) | c \in [2, \lambda]\right\}\right)
\times D^k_{n-1}(\lambda)\right) \rightarrow D^k_n(\lambda)
\]
which is surjective, but give some permutations twice. These
permutations will also be given once by 
\[
\zeta_2: \left([2, \lambda] \times [k + 2, n] \times D^k_{n-2}(\lambda)\right)
\rightarrow D^k_n(\lambda),
\]
thereby proving the proposition.

For $\pi \in D^k_{n-1}(\lambda)$, let $F_j = \{i \in F_\pi | C(i) > 1,
i < j\}$ be the essential fixed points less than $j$ in $\pi$. Then,
$\zeta_1(j, c, \pi) = \phi_{F_j}\circ\phi_{k+1, j-|F_j|}\circ\psi_{F_j}(\pi)$,
which inserts the element $j$ as soon as possible after position $k$, without disturbing the
coloured fixed points. If $j = k+1$, we let $C(k+1) = c$.

The permutations given twice are those where the segment $F$ of essential 
fixed points starting at $k+1$ is non-empty, and followed by
an element above or on the diagonal. For $\pi \in D^k_{n-2}(\lambda)$, these
are given by applying $\zeta_2(c, j, \pi) = \phi_F\circ\phi_{k+2,
j}\circ\phi_{k+1}\circ\psi_F(\pi)$ and $C(k+1) = c$, that is we insert a fixed
point $k+1$ with a non-default colour $c$ followed by a default
coloured element $j$ larger than $k+1$. 
\end{proof}
 
These formulae allow us to once again deduce the recursion for the $\lambda$-factorials. Using Proposition \ref{prop:3} extended to $k = -1$ and $d^{-1}_{-1}(\lambda) = 1$, we get by induction $d^{-1}_n = (\lambda - 1)d^{-1}_{n-1}$ and hence $d^{-1}_n = (\lambda - 1)^{n+1}$. Thus, by Proposition~\ref{prop:2} we have $f_\lambda(n) = d^0_n = n d^0_{n-1} + (\lambda - 1)^n$.
We can also use Proposition \ref{prop:3} to obtain, using (\ref{lamfakrec}),
\[
\begin{split}
f_\lambda\left(n\right) &= (n + \lambda - 1)f_\lambda\left(n-1\right) - (\lambda - 1)
(n - 1)f_\lambda\left(n-2\right)\\
&= (n-1)\left(f_\lambda\left(n-1\right) + f_\lambda\left(n-2\right)\right) + \lambda\left(f_\lambda\left(n-1\right) - (n-1)f_\lambda\left(n-2\right)\right) \\
&= (n-1)\left(f_\lambda\left(n-1\right) + f_\lambda\left(n-2\right)\right) + \lambda (\lambda - 1)^{n-1},
\end{split}
\]
which specialises to the well-known $$D_n = (n - 1)(D_{n-1}
+ D_{n-2})$$ and $n! = (n-1)((n-1)! + (n-2)!)$. 

\begin{example}

We consider the set $D^2_4(2)$, using bold face for the second
colour. In the table below, we give the permutations which are in
bijection with those in $D^2_4(2)$ via $\theta$ and $\eta$, and those
being mapped there by $\zeta_1$. At the star, both $(4, 1, 21\col{3})$
and $(3, 2, 213)$ are mapped to $21\col{3}4$ by $\zeta_1$, as is $(2,
4, 21)$ by $\zeta_2$. 

\begin{center}

\begin{tabular}{|c|c|c|c||c|c|c|c|}
\hline
$D^2_4(2)$ & $\theta$ & $\eta$ & $\zeta_1$ & $D^2_4(2)$ & $\theta$ & $\eta$ & $\zeta_1$ \\
\hline
$2134$       & $(1, 3214)$       & $(3, 213)$       & $(3, 1, 213)$       &    
$4123$       & $(2, 4213)$       & $(2, 312)$       & $(2, 1, 312)$       \\
$21\col{3}4$ & $(2, 213)$        & $(4, 21\col{3})$ & $\star$             &
$4132$       & $(2, 4312)$       & $(3, 312)$       & $(3, 1, 312)$       \\
$213\col{4}$ & $(1, 321\col{4})$ & $(3, 21\col{3})$ & $(3, 1, 21\col{3})$ &
$41\col{3}2$ & $(2, 312)$        & $(2, 12)$        & $(3, 2, 312)$       \\
$21\col{34}$ & $(2, 21\col{3})$  & $(2, 1\col{2})$  & $(3, 2, 21\col{3})$ &
$4213$       & $(3, 4213)$       & $(1, 312)$       & $(1, 1, 312)$       \\
$2143$       & $(1, 4213)$       & $(4, 213)$       & $(4, 1, 213)$       &
$4231$       & $(2, 4321)$       & $(3, 321)$       & $(3, 1, 321)$       \\
$3124$       & $(2, 3214)$       & $(2, 213)$       & $(2, 1, 213)$       &
$42\col{3}1$ & $(2, 321)$        & $(2, 21)$        & $(3, 2, 321)$       \\
$312\col{4}$ & $(2, 321\col{4})$ & $(2, 21\col{3})$ & $(2, 1, 21\col{3})$ &
$4312$       & $(3, 4312)$       & $(1, 321)$       & $(1, 1, 321)$       \\
$3142$       & $(1, 4312)$       & $(4, 312)$       & $(4, 1, 312)$       &
$4321$       & $(3, 4321)$       & $(2, 321)$       & $(2, 1, 321)$       \\
$3214$       & $(3, 3214)$       & $(1, 213)$       & $(1, 1, 213)$       &&&&\\
$321\col{4}$ & $(3, 321\col{4})$ & $(1, 21\col{3})$ & $(1, 1, 21\col{3})$ &&&&\\
$3241$       & $(1, 4321)$       & $(4, 321)$       & $(4, 1, 321)$       &&&&\\
\hline
\end{tabular}

\end{center}

To further examplify the trickiest parts, consider $\eta(1, 542361)$ for
$k = 4$ and $n = 8$. Keeping only the first $k$ elements in $542361$ we
get $4312$ and thus $m = 2$. Returning to $542361$, we sort the first
$k$ elements into $543261$, insert $m$ at position $k+1$, giving
$6543261$, and then the fixed point $k+1$ with colour $2$, giving
$7643\col{5}281$. For the inverse procedure, remove $\col{5}$ and $m =
2$, giving $543261$, and then move the element at position $k - m + 1
= 3$ to position $k = 4$.

\end{example}

We close this section by noting that Lemma \ref{lm:lamfakid} can be
generalised to $d^k_n(\lambda)$ as follows. The proof is completely
analogous. 

\begin{proposition}
For $\nu, \lambda \in \mathbb{C}$ and $0 \leq k \leq n \in \mathbb N$,
we have 
\[
d^k_{n}(\nu) = \sum_j \binom{n-k}{j} d^k_{n - j}(\lambda) (\nu - \lambda)^j.
\]
\end{proposition}

\section{Open problems}

While many of our results have been shown bijectively, there are a few
that still seek their combinatorial explanation. The most obvious are
these. 

\begin{problem}
Give a combinatorial proof, using the principle of
inclusion-exclusion, of Theorem \ref{thm:lamfak}.
\end{problem}

\begin{problem}
Give a bijection $f: \Sn \rightarrow \Sn$ such that $\pi \in D(a_1,
a_2, \ldots a_k) \Rightarrow f(\pi) \in D(a_1 + 1, a_2 - 1, a_3,
\ldots, a_k)$ whenever $a_1 \geq a_2$ and $\vc{a} \neq (2m, 1)$.
\end{problem}

We would also like the rearrangement of blocks in $D(\vc{a})$ to get
a simple description.

\begin{problem}
For any $(a_1, \ldots, a_k)$ and any $\sigma \in \SymGroup{k}$,
give a simple bijection $f: D(a_1, \ldots, a_k) \rightarrow D(a_{\sigma_1},
\ldots, a_{\sigma_k})$.
\end{problem}

Instead of specifying descents, we could specify spots where the permutation must not descend. This would add some new features to the problem, as ascending blocks can contain several fixed points, whereas descending blocks can only contain one.

\begin{problem}
Given a composition $\vc a$, find the number of derangements that ascend within the blocks.
\end{problem}

\bibliographystyle{plain}
\bibliography{DerDes}

\addcontentsline{toc}{chapter}{\numberline{}Bibliography}
\end{document}